\documentclass[10pt]{amsart}
\usepackage{amsmath,amsfonts,amssymb}
\usepackage{enumerate}
\newtheorem{lemma}{Lemma}[section]

\newtheorem{theorem}[lemma]{Theorem}

\theoremstyle{definition}

\theoremstyle{remark}

\newcommand{\sA}{\mathsf{A}}

\numberwithin{equation}{section} \numberwithin{table}{section}

\newcommand{\Z}{\mathbb{Z}}
\newcommand{\supp}{\mathrm{supp}\,}

\begin{document}
\title[Maximising measures for weighted shift operators]{Lyapunov-maximising measures for pairs of weighted shift operators}
\author{Ian D. Morris}
\address{Department of Mathematics, University of Surrey, Guildford GU2 7XH, U.K.}
\email{i.morris@surrey.ac.uk}

\maketitle
\begin{abstract}
Motivated by recent investigations of ergodic optimisation for matrix cocycles, we study the measures of maximum top Lyapunov exponent for pairs of bounded weighted shift operators on a separable Hilbert space. We prove that for generic pairs of weighted shift operators the Lyapunov-maximising measure is unique, and show that there exist pairs of operators whose unique Lyapunov-maximising measure takes any prescribed value less than $\log 2$ for its metric entropy. We also show that in contrast to the matrix case, the Lyapunov-maximising measures of pairs of bounded operators are in general not characterised by their supports: we construct explicitly a pair of operators, and a pair of ergodic measures on the 2-shift with identical supports, such that one of the two measures is Lyapunov-maximising for the pair of operators and the other measure is not. Our proofs make use of the Ornstein $\overline{d}$-metric to estimate differences in the top Lyapunov exponent of a pair of weighted shift operators as the underlying measure is varied.

\end{abstract}

\section{Introduction and principal results}

Let $\mathsf{A}$ be a compact set of $d \times d$ real matrices. The joint spectral radius of $\mathsf{A}$, which we denote by $\varrho(\mathsf{A})$, is defined to be the maximum possible exponential growth rate of products of matrices from the set $\mathsf{A}$
\begin{align*}\varrho(\mathsf{A})&:=\lim_{n \to \infty}\sup\left\{\left\|A_n\cdots A_1 \right\|^{\frac{1}{n}} \colon A_i \in \mathsf{A}\right\}\\
&=\sup_{(A_i)_{i=1}^\infty \in \mathsf{A}^{\mathbb{N}}} \limsup_{n \to \infty}\left\|A_n \cdots A_1\right\|^{\frac{1}{n}},\end{align*}
which is independent of the choice of norm $\|\cdot\|$ on the vector space of $d \times d$ matrices. The identity between the two formulas above is straightforward to prove, and a proof may be found in, for example, \cite{Ju09}; one may show in particular that the suprema in the above expressions are always attained. It is then natural to ask precisely \emph{which} sequences $(A_i)\in\mathsf{A}^{\mathbb{N}}$ attain the second of these suprema; we will call these sequences \emph{maximising sequences} for $\sA$. This issue was first explicitly raised by J. Lagarias and Y. Wang \cite{LaWa95} and independently by L. Gurvits \cite{Gu95} in 1995, in which it was asked (in equivalent formulations) whether there always exists a periodic maximising sequence. The answer to this question is now known to be negative (\cite{BoMa02}, see also \cite{BlThVl03,Bo08,HaMoSiTh11,JePo15,Ko07}), but the set of possible structures of maximising sequences for a set of matrices remains elusive.

It turns out that insight into the structure of extremal sequences can be gained by replacing the consideration of \emph{sequences} with consideration of \emph{measures}. If $(A_i)_{i=1}^\infty \in \mathsf{A}^{\mathbb{Z}}$ is a maximising sequence for $\sA$ then it is easy to show that the shifted sequence $(A_{i+1})_{i=1}^\infty \in \mathsf{A}^{\mathbb{Z}}$ is maximising also, so the set of maximising sequences is a shift-invariant subset of $\mathsf{A}^{\mathbb{Z}}$. It is therefore perhaps unsurprising that maximising sequences may be broadly characterised in terms of shift-invariant measures. To describe this characterisation let us fix some notation and terminology. Given a compact set $\sA$ of $d \times d$ real  matrices  let $\sigma \colon \sA^\Z \to \sA^\Z$ denote the shift transformation $(A_i)_{i=1}^\infty \mapsto (A_{i+1})_{i=1}^\infty$, which is a continuous transformation of the compact metrisable topological space $\sA^\Z$. The set $\sA$ being understood, let $\mathcal{M}_\sigma$ denote the set of all $\sigma$-invariant Borel probability measures on $\sA^\Z$. We equip $\mathcal{M}_\sigma$ with the weak-* topology which it inherits as a closed subset of $C(\sA^\Z)^*$, and with respect to this topology it is a compact metrisable topological space. For each $\mu \in \mathcal{M}_\sigma$ we define
\[\Lambda(\sA,\mu):=\lim_{n \to \infty}\frac{1}{n}\int \log\|A_n\cdots A_1\|d\mu[(A_i)] = \inf_{n \to \infty}\frac{1}{n}\int \log\|A_n\cdots A_1\|d\mu[(A_i)].\]
If $\mu$ is ergodic then by the subadditive ergodic theorem
\[\lim_{n \to \infty}\frac{1}{n}\log \|A_n\cdots A_1\| = \Lambda(\sA,\mu)\]
for $\mu$-almost-every sequence $(A_i)_{i=1}^\infty \in \sA^\Z$. If $\mu \in \mathcal{M}_\sigma$ and $\Lambda(\sA,\mu)=\log\varrho(\sA)$ then it follows from this (together with the ergodic decomposition theorem) that $\mu$-almost-every sequence in $\sA^\Z$ is a maximising sequence. For this reason if $\Lambda(\sA,\mu)=\log\varrho(\sA)$ we call the measure $\mu$ a \emph{maximising measure} for $\sA$.

A proof of the following relatively simple result was given previously as Theorem 2.1 and Proposition 2.2 in \cite{Mo13}, but this idea was used implicitly by T. Bousch and J. Mairesse in \cite{BoMa02}:
\begin{theorem}
Let $\sA$ be a compact set of $d \times d$ real  matrices with nonzero joint spectral radius. Then we have
\[\log\varrho(\sA)=\sup_{\mu \in \mathcal{M}_\sigma} \inf_{n \geq1}\frac{1}{n}\int_{\sA^\Z}\log\|A_n\cdots A_1\|d\mu[(A_i)].\]
The set of all maximising measures of $\mathsf{A}$ is a nonempty, compact, convex subset of $\mathcal{M}_\sigma$, and its extremal points are precisely its ergodic elements.
\end{theorem}
We remark that the exceptional condition $\varrho(\sA)=0$ can arise only under extremely restricted circumstances: if $\varrho(\sA)=0$ then the elements of $\sA$ can be simultaneously conjugated to upper-triangular matrices with zero diagonal, see for example \cite{Ju09}.
The structure of the set of maximising measures and its relationship with maximising sequences is further elucidated in the following result (\cite[Theorem 2.3]{Mo13}):
\begin{theorem}\label{th:mather}
Let $\sA$ be a compact set of $d \times d$ real matrices. Then there exists a nonempty closed $\sigma$-invariant set $Z \subseteq \sA^\Z$ with the following properties:
\begin{enumerate}[(i)]
\item
For every $\mu \in \mathcal{M}_\sigma$, the measure $\mu$ is maximising for $\mathsf{A}$ if and only if $\mu(Z)=1$.
\item
Every sequence $(A_i) \in Z$ is a maximising sequence.
\item
If $(A_i)\in\mathsf{A}^{\mathbb{Z}}$ is a maximising sequence and $f \colon \mathsf{A}^{\mathbb{Z}}\to\mathbb{R}$ is a continuous function which is identically zero on $Z$, then $\frac{1}{n}\sum_{k=0}^{n-1}f(\sigma^k[(A_i)]) \to 0$.
\end{enumerate}
\end{theorem}
Clause (i) of the above theorem is analogous to the ``subordination principle'' in ergodic optimisation noted by T. Bousch \cite{Bo01}. If $T \colon X \to X$ is a continuous transformation of a compact metric space, $f \colon X \to \mathbb{R}$ a continuous function and $\mathcal{M}_T$ the set of $T$-invariant Borel probability measures on $X$, then the maximising measures of $f$ are defined to be those elements of $\mathcal{M}_T$ which achieve the supremum $\sup_{\mu\in\mathcal{M}_T}\int f\,d\mu$. Under suitable regularity conditions on $T$ and $f$, these measures are characterised among the $T$-invariant measures by their being supported in a $T$-invariant set $Z\subseteq X$ in a manner similar to (i) above: see for example \cite{Bo01,CoLoTh01,Mo07}. It is known that the subordination principle may be violated if the regularity of $f$ or $T$ is sufficiently poor, but so far the only known examples of this phenomenon are given by nonconstructive existence proofs \cite{BoJe02,Je06b}.

Continuing the line of inquiry started by Lagarias, Wang and Gurvits it is natural to ask if the set of possible maximising sequences of a matrix set admits any stronger characterisation than that imposed by Theorem \ref{th:mather}. Restricting for a moment to the case where $\sA$ is a pair of matrices, we may naturally identify $\sA^\Z$ with the space $\Sigma_2:=\{0,1\}^\Z$. In this framework the following question becomes natural: which sets of invariant measures on $\Sigma_2$ which are consistent with Theorem \ref{th:mather}(i) correspond to the set of maximising measures of some pair of matrices $\sA$? For example, if the support of a measure $\mu$ on $\Sigma_2$ is uniquely ergodic, can $\{\mu\}$ be the set of maximising measures of some pair of matrices $\sA$? Is this possible  if $\mu$ has nonzero entropy? Does a generic pair of matrices $\sA$ have a unique maximising measure? While some useful information has been obtained for specific classes of matrices (see for example \cite{BoRa13,HaMoSi13,JePo15,Mo10b}) general answers to these questions currently seem distant.

When compared to similar questions in ergodic optimisation (see e.g. \cite{Bo08,Co13,Mo08,YuHu99}) the investigation of generic properties of maximising measures of matrix sets is hampered by the difficulty of constructing advantageous perturbations in a finite-dimensional setting. In this article, inspired by earlier work of L. Gurvits \cite{Gu95}, we investigate a related, simpler problem by studying the maximising sequences and maximising measures of pairs of weighted shift operators (defined below) on a separable real Hilbert space. In this context the infinite number of degrees of freedom makes generic phenomena  much easier to study, while at the same time the simple dynamical behaviour of weighted shift operators permits us very strong control on the growth rates of their products. This allows us to give answers to the analogues of the previous questions in the context of this class of operators on Hilbert spaces.

%In this article we shall show that for a generic pair of weighted shift operators there exists a unique maximising measure, and for every strictly ergodic measure on $\Sigma_2$ we are able to contruct a pair of weighted shift operators for which that measure is the unique maximising measure. On the other hand, we are able to exhibit pairs of weighted shift operators for which no set having the properties (i)--(iii) of Theorem \ref{th:mather} exists. These results are formally stated in the following section. With reference to these ideas we also indicate in \S3 how a method used by B. Kalinin in \cite{Ka11} to prove the Liv\v{s}ic periodic orbit criterion for $GL_d(\R)$-cocycles cannot be generalised to the case of operator-valued cocycles.

Here and throughout the paper, let $\mathcal{H}$ denote a real separable Hilbert space spanned by the orthonormal basis $\{e_i \colon i \in \mathbb{Z}\}$. We let $\mathcal{B}(\mathcal{H})$ denote the set of all bounded linear operators on $\mathcal{H}$, and we equip this set with the operator norm topology. A \emph{weighted shift operator} on $\mathcal{H}$ is a bounded linear operator $L \colon \mathcal{H} \to \mathcal{H}$ such that for each basis element $e_i$, the vector $Le_i$ is proportional to $e_{i+1}$. If $L$ is a weighted shift operator such that $Le_i=\alpha_i e_{i+1}$ for all $i \in \Z$, we may directly compute that $\|L\|=\sqrt{\rho(L^*L)}=\sup_{i \in \Z} |\alpha_i|$, where $\rho$ denotes spectral radius. More generally, if $L_1,\ldots,L_n$ are weighted shift operators such that $L_ke_i=\alpha_{k,i}e_{i+1}$ for all $i \in \Z$ and $k=1,\ldots,n$ then we have $\|L_n\cdots L_1\|=\sup_{i \in \Z}\left|\prod_{k=1}^n \alpha_{k,i+k-1}\right|$. Since our interest in weighted shift operators is restricted to investigating the norms of their compositions we shall incur no loss of generality in considering only those operators for which the coefficients $\alpha_i$ are real and non-negative. Since furthermore we will lose no generality in our conclusions by considering only \emph{invertible} weighted shift operators, we shall assume that the sequences $(\alpha_i)$ are bounded away from zero. We let $\mathcal{W} \subset \mathcal{B}(\mathcal{H})^2$ denote the set of all pairs of bounded, invertible weighted shift operators with positive coefficients. 
 It is clear that the set of all weighted shift operators is a Banach subspace of $\mathcal{B}(\mathcal{H})$ and that the set of all invertible weighted shift operators with positive coefficients is an open subset of that space, so in particular $\mathcal{W}$ is an open subset of a Banach space and hence is a Baire space. Given a pair $(L_0,L_1)\in\mathcal{W}$ we identify $\{L_0,L_1\}^{\mathbb{Z}}$ with $\Sigma_2$ and the set of shift-invariant Borel measures on $\{L_0,L_1\}^{\mathbb{Z}}$ with that on $\Sigma_2$ in the obvious fashion. We prove the following results:

\begin{theorem}\label{th:strictly}
Let $\mu$ be a measure on $\Sigma_2$ whose support is uniquely ergodic. Then there exists an open set $\mathcal{U}\subset \mathcal{W}$ such that for every pair of weighted shift operators $(L_0,L_1) \in \mathcal{U}$, $\mu$ is the unique maximising measure of $(L_0,L_1)$. 
\end{theorem}
For every $h \in [0,\log 2)$ there exists a measure on $\Sigma_2$ with entropy $h$ whose support is uniquely ergodic (see e.g. \cite{DeGrSi76,Gr72}), so Theorem \ref{th:strictly} implies in particular that for every $h \in [0,\log 2)$ there exists a pair $(L_0,L_1) \in \mathcal{W}$ whose sole maximising measure has entropy precisely $h$.

In the context of weighted shift operators, generic uniqueness of the maximising measure turns out to be more easily established than in the matrix case:
\begin{theorem}\label{th:unique}
The set of all pairs $(L_0,L_1) \in \mathcal{W}$ having a unique maximising measure is a dense $G_\delta$ subset of $\mathcal{W}$.
\end{theorem}
On the other hand, in the infinite-dimensional case the analogue of Theorem \ref{th:mather}(i) is false:
\begin{theorem}\label{th:nomather}
There exist a pair $(L_0,L_1)\in \mathcal{W}$ and two shift-invariant measures $\mu,\nu$ on $\Sigma_2$ such that $\mu$ and $\nu$ have the same support, but $\mu$ is maximising for $\{L_0,L_1\}$ and $\nu$ is not.
\end{theorem} 
The proof of Theorem \ref{th:nomather} is explicit and constructive. While T. Bousch and O. Jenkinson have shown that the subordination principle in classical ergodic optimisation can fail to hold in certain cases (\cite{BoJe02,Je06b}), all known examples are nonconstructive, relying on Baire's theorem or the Hahn-Banach theorem. Theorem \ref{th:nomather} might therefore reasonably be described as the most natural known example of failure of the subordination principle.

The remainder of this article is structured as follows.  In \S\ref{se:part2} we indicate the mechanisms by which we shall transform the investigation of pairs of weighted shifted operators into the direct study of their sequences of weights. In that section we also fix some notational conventions which will be used throughout the remainder of the paper, and we recall the definition of the Ornstein $\overline{d}$-metric, which plays an essential role in many subsequent arguments. In \S\ref{se:foundation} we prove some simple but essential lemmas which underpin the proofs of the main theorems; Theorems \ref{th:strictly}--\ref{th:nomather} themselves are proved in \S\ref{se:strictly}--\ref{se:lemmer}. 

%%%%%%%%%%%%%%%%%%%%%%%%%%%
%                  NEW SECTION
%%%%%%%%%%%%%%%%%%%%%%%%%%%

\section{Reduction to sequences of weights}\label{se:part2}

Let $\Sigma_2:=\{0,1\}^{\mathbb{Z}}$ with the infinite product topology, and define $\sigma \colon \Sigma_2 \to \Sigma_2$ by $\sigma[(x_i)_{i\in\mathbb{Z}}]:=(x_{i+1})_{i \in \mathbb{Z}}$. The transformation $\sigma$ is a homeomorphism, and $\Sigma_2$ is compact and metrisable. Throughout this article we will frequently use the shorthand $\mathbb{Z}_2:=\{0,1\}$. If $\phi \in \ell_\infty(\mathbb{Z}_2\times\mathbb{Z})$ then we define a pair of weighted shift operators $(L_0^\phi,L_1^\phi)$ by
\[L_a^\phi e_i:=e^{\phi(a,i)}e_{i+1}\]
for all $i \in \mathbb{Z}$ and $a \in \mathbb{Z}_2$. For every $x=(x_i)_{i \in \mathbb{Z}} \in \Sigma_2$ and every $n \geq 1$ we have
\begin{equation}\label{eq:common}\left\|L_{x_n}^\phi\cdots L_{x_1}^\phi \right\|=\exp\left(\sup_{k \in \mathbb{Z}}\sum_{i=0}^{n-1}\phi(x_i,i+k)\right),\end{equation}
an identity which will frequently be used without comment. We leave it to the reader to verify the following simple result:
\begin{lemma}
The map $\ell_\infty(\mathbb{Z}_2\times\mathbb{Z}) \to \mathcal{W}$ defined by $\phi \mapsto (L^\phi_0,L^\phi_1)$ is a homeomorphism.
\end{lemma}
We will establish Theorems \ref{th:strictly}, \ref{th:unique} and \ref{th:nomather} by proving corresponding statements concerning the pairs $(L_0^\phi,L_1^\phi)$ for various $\phi \in \ell_\infty(\mathbb{Z}_2\times\mathbb{Z})$.

For the remainder of the article we let $\mathcal{M}_\sigma$ denote the set of all $\sigma$-invariant Borel probability measures on $\Sigma_2$. Except when indicated otherwise, we equip $\mathcal{M}_\sigma$ with the weak-* topology which it inherits by identification with the positive unit sphere of $C(\Sigma_2)^*$; this is defined to be the smallest topology on $\mathcal{M}_\sigma$ such that $\mu \mapsto \int f\,d\mu$ is continuous for every $f \in C(\Sigma_2)$. The set $\mathcal{M}_\sigma$ thus equipped is a compact, convex, metrisable subspace of the locally convex topological space $C(\Sigma_2)^*$. A sequence of measures $(\mu_n) \in \mathcal{M}_\sigma$ converges to a limit $\mu \in \mathcal{M}_\sigma$ if and only if $\int f\,d\mu_n \to \int f\,d\mu$ for every $f \in C(\Sigma_2)$.
We let $\mathcal{E}_\sigma$ denote the set of ergodic elements of $\mathcal{M}_\sigma$. 

Given $\phi \in \ell_\infty(\mathbb{Z}_2\times\mathbb{Z})$ and $\mu \in \mathcal{M}_\sigma$, we define
\[\Lambda(\phi,\mu):=\lim_{n \to \infty} \frac{1}{n}\int \log\left\|L^\phi_{x_n}\cdots L^\phi_{x_1}\right\|d\mu(x) =\inf_{n \geq1} \frac{1}{n}\int \log\left\|L^\phi_{x_n}\cdots L^\phi_{x_1}\right\|d\mu(x),\]
where the existence of the limit and its identity with the above infimum arise from subadditivity. In view of \eqref{eq:common} we equivalently have
\[\Lambda(\phi,\mu)=\lim_{n \to \infty} \frac{1}{n}\int \sup_{k \in \mathbb{Z}}\sum_{i=0}^{n-1}\phi(x_i,i+k)d\mu(x) =\inf_{n \geq1} \frac{1}{n}\int \sup_{k \in \mathbb{Z}}\sum_{i=0}^{n-1}\phi(x_i,i+k)d\mu(x).\]
When $\mu\in\mathcal{E}_\sigma$ we find by the subadditive ergodic theorem that for $\mu$-a.e. $x \in \Sigma_2$,
\[\Lambda(\phi,\mu)=\lim_{n \to \infty}\frac{1}{n}\log\left\|L^\phi_{x_n}\cdots L^\phi_{x_1}\right\|=\lim_{n \to \infty}\frac{1}{n}\sup_{k \in \mathbb{Z}}\sum_{i=0}^{n-1}\phi(x_i,i+k).\]
The joint spectral radius $\varrho(L_0^\phi,L_1^\phi)$ of the pair of weighted shift operators $L_0^\phi,L_1^\phi$ is defined by
\[\varrho\left(L_0^\phi,L_1^\phi\right):=\lim_{n \to \infty}\sup_{x_1,\ldots,x_n \in\{0,1\}}\left\|L_{x_n}^\phi\cdots L_{x_1}^\phi \right\|^{\frac{1}{n}}=\inf_{n \geq 1}\sup_{x_1,\ldots,x_n \in\{0,1\}}\left\|L_{x_n}^\phi\cdots L_{x_1}^\phi \right\|^{\frac{1}{n}}\]
which is clearly positive and satisfies
\[\log \varrho\left(L_0^\phi,L_1^\phi\right)=\lim_{n \to \infty} \sup_{x \in \Sigma_2} \sup_{k \in \mathbb{Z}} \sum_{i=0}^{n-1}\phi(x_i,i+k).\]
We recall the following theorem of S.J. Schreiber, R. Sturman and J. Stark \cite{Sc98,StSt00}:
\begin{theorem}[Schreiber, Sturman-Stark]\label{th:stst}
Let $T \colon X \to X$ be a continuous transformation of a compact metric space, and let $(f_n)$ be a sequence of continuous functions $f_n \colon X \to \mathbb{R}$ such that $f_{n+m}(x)\leq f_n(T^mx)+f_m(x)$ for all $x \in X$ and $n,m \geq 1$. Let $\mathcal{M}_T$ denote the set of all $T$-invariant Borel probability measures on $X$; then
\[\lim_{n \to \infty} \sup_{x \in X} \frac{1}{n}f_n(x)=\sup_{\mu \in\mathcal{M}_T}\lim_{n \to \infty}\frac{1}{n}\int f_n\,d\mu.\]
\end{theorem}
Applying this result to the transformation $\sigma \colon \Sigma_2 \to \Sigma_2$ and the sequence of functions $f_n(x):=\log\|L^\phi_{x_n}\cdots L^\phi_{x_1}\|$ we find that 
\[\log\varrho\left(L_0^\phi,L_1^\phi\right)= \sup_{\mu \in\mathcal{M}_\sigma} \Lambda(\phi,\mu)\]
for every $\phi \in \ell_\infty(\mathbb{Z}_2\times\mathbb{Z})$. For each $\phi \in\ell_\infty(\mathbb{Z}_2\times\mathbb{Z})$ let us therefore define
\[\mathcal{M}_{\max}(\phi):=\left\{\mu \in \mathcal{M}_\sigma \colon \Lambda(\mu,\phi)=\log\varrho\left(L_0^\phi,L_1^\phi\right)\right\};\]
this coincides with the set of maximising measures of $(L_0^\phi,L_1^\phi)$ in the sense of the previous section. It follows from \cite[Proposition A.5]{Mo13} that for every $\phi \in \ell_\infty(\mathbb{Z}_2\times\mathbb{Z})$, the set $\mathcal{M}_{\max}(\phi)$ is a nonempty compact convex subset of $\mathcal{M}_\sigma$ which is equal to the closed convex hull of $\mathcal{M}_{\max}(\phi)\cap \mathcal{E}_\sigma$.

In this article we will find it convenient to additionally  consider a metric on the set $\mathcal{M}_\sigma$ which generates a much stronger topology, namely the Ornstein $\overline{d}$-metric. Let $\pi_1,\pi_2 \colon \Sigma_2 \times \Sigma_2 \to \Sigma_2$ be given by projection onto the first and second co-ordinates respectively, and recall that the push-forward measures $(\pi_1)_*m, (\pi_2)_*m$ of a measure $m$ on $\Sigma_2\times \Sigma_2$ are the measures on $\Sigma_2$ defined by $(\pi_i)_*(A):= m(\pi_i^{-1}A)$ for $i=1,2$ and Borel $A\subseteq \Sigma_2$. Define a function $\mathfrak{d}\colon \Sigma_2 \times \Sigma_2 \to \mathbb{R}$ by $\mathfrak{d}(x,y)=0$ if $x_0=y_0$, and $\mathfrak{d}(x,y)=1$ if $x_0 \neq y_0$. Given two measures $\mu_1,\mu_2 \in \mathcal{M}_\sigma$ we say that a measure $m$ on $\Sigma_2 \times \Sigma_2$ is a \emph{joining} of $\mu_1$ with $\mu_2$ if it is $(\sigma\times\sigma)$-invariant and satisfies $(\pi_1)_*m=\mu_1$ and $(\pi_2)_*m=\mu_2$, and we write $\mathcal{J}(\mu_1,\mu_2)$ for the set of all such measures $m$.  Given $\mu_1,\mu_2 \in \mathcal{M}_\sigma$ we then define
\[\overline{d}(\mu_1,\mu_2):=\inf_{m \in \mathcal{J}(\mu_1,\mu_2)} \int \mathfrak{d}\,dm.\]
The function $\overline{d}$ defines a metric on $\mathcal{E}_\sigma$ with respect to which $\mathcal{E}_\sigma$ is complete but not separable (see e.g. \cite[\S I.9]{Sh96}, \cite[\S15.7]{Gl03}). The topology generated by $\overline{d}$ refines the weak-* topology: weak-* closed subsets of $\mathcal{E}_\sigma$ are also closed with respect to $\overline{d}$.

\section{Some fundamental observations}\label{se:foundation}

In this section we prove some simple observations on the objects $L^\phi_i$ and $\mathcal{M}_{\max}(\phi)$ defined above which will be useful in the remainder of the article.
\begin{lemma}\label{le:simples}
Let $\phi_1,\phi_2 \in \ell_\infty(\mathbb{Z}_2\times\mathbb{Z})$. Then
\[\Lambda(\phi_1+\phi_2,\mu)\leq \Lambda(\phi_1,\mu)+\Lambda(\phi_2,\mu)\]
for every $\mu\in\mathcal{M}_\sigma$.
\end{lemma}
\begin{proof}
For each $x \in \Sigma_2$, $n \geq 1$ and $\ell \in \mathbb{Z}$ we have
\begin{align*}\sum_{i=0}^{n-1}(\phi_1+\phi_2)(x_i,i+\ell) &=\sum_{i=0}^{n-1}\phi_1(x_i,i+\ell) +\sum_{i=0}^{n-1}\phi_2(x_i,i+\ell)  \\
&\leq \sup_{k \in \mathbb{Z}}\sum_{i=0}^{n-1}\phi_1(x_i,i+k)+ \sup_{k \in \mathbb{Z}}\sum_{i=0}^{n-1}\phi_2(x_i,i+k),\end{align*}
so taking the supremum over $\ell \in \mathbb{Z}$ we obtain
\[\sup_{k \in \mathbb{Z}}\sum_{i=0}^{n-1}(\phi_1+\phi_2)(x_i,i+k)\leq \sup_{k \in \mathbb{Z}}\sum_{i=0}^{n-1}\phi_1(x_i,i+k)+ \sup_{k \in \mathbb{Z}}\sum_{i=0}^{n-1}\phi_2(x_i,i+k)\]
and thus
\[\log\left\|L^{\phi_1+\phi_2}_{x_n}\cdots L^{\phi_1+\phi_2}_{x_1}\right\| \leq \log\left\|L^{\phi_1}_{x_n}\cdots L^{\phi_1}_{x_1}\right\| +\log\left\|L^{\phi_2}_{x_n}\cdots L^{\phi_2}_{x_1}\right\|\]
for every $x \in \Sigma_2$ and $n \geq 1$. Integrating over $x$ with respect to $\mu$, dividing by $n$ and taking the limit as $n\to\infty$ yields the result.
\end{proof}
\begin{lemma}\label{le:cty}
The function $ \ell_\infty(\mathbb{Z}_2\times\mathbb{Z}) \to \mathbb{R}$ defined by $\phi \mapsto \log\varrho(L_0^\phi,L_1^\phi)$ is $1$-Lipschitz continuous.
\end{lemma}
\begin{proof}
By symmetry it is enough to show that for all $\phi_1,\phi_2 \in \ell_\infty(\mathbb{Z}_2\times\mathbb{Z})$ we have \[\log\varrho(L_0^{\phi_1},L_1^{\phi_1})\leq \log\varrho(L_0^{\phi_2},L_1^{\phi_2})+|\phi_1-\phi_2|_\infty.\]
Let $\mu \in \mathcal{M}_{\max}(\phi_1)$. The inequality $\Lambda(\phi,\mu)\leq |\phi|_\infty$ is obvious. Combining this with Lemma \ref{le:simples} we have
\begin{align*}\log \varrho\left(L_0^{\phi_1},L_1^{\phi_1}\right) &= \Lambda(\phi_1,\mu) \leq \Lambda(\phi_2,\mu)+\Lambda(\phi_1-\phi_2,\mu)\\
&\leq \Lambda(\phi_2,\mu) + |\phi_1-\phi_2|_\infty\\
&\leq \log\varrho\left(L_0^{\phi_2},L_1^{\phi_2}\right)+|\phi_1-\phi_2|_\infty\end{align*}
as required.
\end{proof}
The following lemma is analogous to \cite[Lemma 6.3]{Mo13} in the matrix case:
\begin{lemma}\label{kuratowski}
Suppose that $(\phi_n)$ is a sequence of elements of $\ell_\infty(\mathbb{Z}_2 \times \mathbb{Z})$ which converges to a limit $\phi$. If $(\mu_n)$ is a sequence of measures on $\Sigma_2$ such that $\mu_n \in \mathcal{M}_{\max}(\phi_n)$ for every $n \geq 1$, and $(\mu_n)$ converges in the weak-* topology to a limit $\mu \in \mathcal{M}_\sigma$, then $\mu \in \mathcal{M}_{\max}(\phi)$.
\end{lemma}
\begin{proof}
For each $m \geq 1$, $\psi\in\ell_\infty(\mathbb{Z}_2\times\mathbb{Z})$ and $x \in \Sigma_2$, let us for notational simplicity define
\[f_m^\psi(x):=\sup_{k \in \mathbb{Z}} \sum_{i=0}^{m-1} \psi(x_i,k+i).\]
It is easily verified that $|f^{\phi_n}_m-f^\phi_m|_\infty\leq m|\phi_n-\phi|_\infty$  for every $n,m \geq 1$, and so for each $m \geq 1$ we have
\begin{align*}\left|\int f_m^\phi\,d\mu - \int f_m^{\phi_n} d\mu_n\right| &\leq \left|\int f_m^\phi\,d\mu - \int f_m^\phi d\mu_n\right|+\left|\int f_m^\phi\,d\mu_n - \int f_m^{\phi_n} d\mu_n\right|\\
&\leq  \left|\int f_m^\phi\,d\mu - \int f_m^\phi d\mu_n\right|+m|\phi-\phi_n|_\infty\end{align*}
for all $n \geq 1$. Each $f_m^\phi \colon \Sigma_2 \to \mathbb{R}$ is continuous since $f^\phi_m$ depends only on the symbols $x_0,\ldots,x_{m-1}$ of $x \in \Sigma_2$. Since $\phi_n \to \phi$ uniformly and $\mu_n \to \mu$ in the weak-* topology  it follows that for each fixed $m$
\[\lim_{n \to \infty} \frac{1}{m}\int f_m^{\phi_n} d\mu_n=\frac{1}{m}\int f_m^\phi\,d\mu.\]% \geq \inf_{k \geq 1}\frac{1}{k}\int f_k^\phi d\mu =\Lambda(\phi,\mu).\]
Using the fact that $\mu_n \in \mathcal{M}_{\max}(\phi_n)$ for every $n \geq 1$ together with the result of Lemma \ref{le:cty} we deduce that the inequality
\begin{align*}\frac{1}{m}\int f_m^\phi\,d\mu &= \lim_{n \to \infty} \frac{1}{m}\int f_m^{\phi_n} d\mu_n\\
& \geq \liminf_{n \to \infty} \inf_{k \geq 1}\frac{1}{k}\int f_k^{\phi_n}\,d\mu_n\\
&=\liminf_{n \to \infty} \Lambda(\phi_n,\mu_n)\\
&=\lim_{n \to \infty} \log\varrho\left(L_0^{\phi_n},L_1^{\phi_n}\right)\\
&=\log\varrho\left(L_0^{\phi},L_1^{\phi}\right)\end{align*}
holds for every integer $m \geq 1$. We conclude that
\[\log\varrho\left(L_0^{\phi},L_1^{\phi}\right) \geq \Lambda(\phi,\mu)=\inf_{m \geq 1}\frac{1}{m}\int f_m^\phi\,d\mu \geq \log\varrho\left(L_0^{\phi},L_1^{\phi}\right)\]
and therefore $\mu \in \mathcal{M}_{\max}(\phi)$ as claimed.
\end{proof}
%%%%%%%%%%%%%%%%%%%%%%%%%%%
%                  NEW SECTION
%%%%%%%%%%%%%%%%%%%%%%%%%%%

\section{Proof of Theorem \ref{th:strictly}}\label{se:strictly}

We will deduce Theorem \ref{th:strictly} from the following somewhat stronger result:
\begin{theorem}\label{th:tech-strictly}
Let $Z\subseteq \Sigma_2$ be a nonempty compact $\sigma$-invariant set, and $z \in Z$ a point such that $\{\sigma^nz \colon n \in\mathbb{Z}\}$ is dense in $Z$. Suppose furthermore that $\phi \in \ell_\infty(\mathbb{Z}_2 \times \mathbb{Z})$ has the following property: for some real number $\delta>0$, we have $\phi(z_i,i)\geq \phi(1-z_i,i)+\delta$ for every $i \in \mathbb{Z}$. Then for every $\mu \in\mathcal{E}_\sigma$ we have
\[\Lambda(\phi,\mu) \leq \log\varrho\left(L_0^{\phi},L_1^{\phi}\right) -\delta \inf\{\overline{d}(\mu,\nu) \colon \nu\in\mathcal{M}_\sigma\text{ and }\nu(Z)=1\}.\]
In the special case where $\phi(z_i,i)=1$, $\phi(1-z_i,i)=0$ for all $i \in\mathbb{Z}$, then for every $\mu \in \mathcal{E}_\sigma$ we  have the exact formula
\[\Lambda(\phi,\mu) = 1 - \inf\{\overline{d}(\mu,\nu) \colon \nu\in\mathcal{E}_\sigma\text{ and }\nu(Z)=1\}.\]
\end{theorem}
The usefulness to us of Theorem \ref{th:tech-strictly} can be captured in the following informal description: if the pair $L_0^\phi,L_1^\phi$ admits a ``uniformly greedy sequence'' $z \in \Sigma_2$ with the property that $\|L^\phi_{z_n}e_n\|>\|L^\phi_{1-z_n}e_n\|$ for all $n \in \mathbb{Z}$, with a uniform lower bound for the ratio of the two quantities, then all maximising measures for $\phi$ must be supported in the orbit closure of the uniformly greedy sequence $z$. Note that the removal of uniformity renders the conclusion false, for example if $L_0^\phi e_n \equiv e_{n+1}$ and $L_1^\phi e_n \equiv 2^{-2^{-|n|}}e_{n+1}$ then the constant sequence of zeros is greedy in the sense that $\|L^\phi_{z_n}e_n\|>\|L^\phi_{1-z_n}e_n\|$ for all $n \in \mathbb{Z}$, but is not \emph{uniformly} greedy in the above sense, and for this pair every measure is maximising.

When the orbit closure of the greedy sequence $z$ is very small, Theorem \ref{th:tech-strictly} conveys precise information on the maximising measures of $\phi$. In more general cases -- such as when the orbit of the greedy sequence is dense in $\Sigma_2$ -- it is unclear how maximising measures are selected, and this matter may be of interest as a topic of future research.

Theorem \ref{th:tech-strictly} may be seen as an extension of the following earlier construction due to L. Gurvits \cite[Theorem A.1]{Gu95}:
\begin{theorem}[Gurvits]
Let $z \in \Sigma_2$ be defined by $z_n:= \chi_{[0,1/2]}(n\gamma -\lfloor n\gamma\rfloor)$ where $\gamma$ is an irrational number. Let $\alpha \in (0,1)$ and define weighted shift operators $L_0$, $L_1$ by $L_0e_n:=\alpha^{1-z_n}e_{n+1}$, $L_1e_n:=\alpha^{z_n}e_{n+1}$ for all $n \in\mathbb{Z}$. Then
\begin{equation}\label{eq:gurv}\varrho(L_0,L_1)=1>\alpha^{\frac{1}{2}} = \sup_{n \geq 1} \max\left\{\rho\left(L_{x_n}\cdots L_{x_1}\right)^{\frac{1}{n}}\colon x_1,\ldots,x_n \in \{0,1\}\right\}\end{equation}
\end{theorem}
Gurvits' result illustrates a contrast with the matrix case, in which the equation
\[\varrho(\mathsf{A})= \sup_{n \geq 1} \sup\left\{\rho\left(A_n\cdots A_1\right)^{\frac{1}{n}}\colon A_1,\ldots,A_n \in \mathsf{A}\right\}\]
is well known to hold for all bounded sets of $d \times d$ matrices $\mathsf{A}$ (see e.g. \cite{BeWa92,Bo03,El95}).  Key to Gurvits' argument is the observation that the sequence $z$ defined above generates a measure $\mu$ on $\Sigma_2$ such that $\overline{d}(\mu,\nu)=\frac{1}{2}$ for every $\nu \in \mathcal{E}_\sigma$ which is supported on a periodic orbit (cf. \cite[p.102]{Sh96}), although Gurvits does not express this observation in terms of ergodic theory. Using Theorem \ref{th:tech-strictly} it is possible to produce further examples of Gurvtis' type: for example, if the orbit closure of $z$ supports a unique measure $\mu$, and this measure has positive entropy, then an identity analogous to \eqref{eq:gurv} must hold, possibly with a smaller exponent than $\frac{1}{2}$. Specifically, if $\mu$ has positive entropy and $\overline{d}(\mu,\nu)$ is sufficiently small then necessarily $h(\nu)>0$, and it follows that there exists $\kappa>0$ such that $\overline{d}(\mu,\nu)\geq \kappa$ whenever $\nu \in \mathcal{E}_\sigma$ is supported on a periodic orbit; this leads to \eqref{eq:gurv} with $\alpha^\kappa$ in place of $\alpha^{1/2}$.

Before giving the proof of Theorem \ref{th:tech-strictly}, let us first verify that it implies Theorem \ref{th:strictly}:
\begin{proof}[Proof of Theorem \ref{th:strictly} conditional on Theorem \ref{th:tech-strictly}]:
Let $\mu \in \mathcal{E}_\sigma$ and suppose that the restriction of $\sigma$ to $Z:=\supp \mu$ is uniquely ergodic. Let $z \in Z$ be arbitrary; it is well-known that $\{\sigma^n z\colon n \in\mathbb{Z}\}$ is necessarily dense in $Z$. Define  $\phi\in\ell_\infty(\mathbb{Z}_2\times\mathbb{Z})$ by $\phi(z_i,i):=1$, $\phi(1-z_i,i):=0$ for every $i \in \mathbb{Z}$, and let $(L_0^\phi,L_1^\phi)\in\mathcal{W}$ be the pair of operators associated to the function $\phi$. 

Since $\mathcal{W}$ is homeomorphic to $\ell_\infty(\mathbb{Z}_2\times\mathbb{Z})$  we may choose an open neighbourhood $\mathcal{U}$ of $(L_0^\phi,L_1^\phi)$ such that every pair in $\mathcal{U}$ has the form $L_0^{\hat\phi},L_1^{\hat\phi}$ where $|\hat\phi-\phi|_\infty<\frac{1}{3}$. For such a pair we have $ \hat\phi(z_i,i)\geq \hat\phi(1-z_i,i)+1/3$ for all $i \in \mathbb{Z}$. Since $\mu$ is the unique element of $\mathcal{E}_\sigma$ which is supported on $Z$, Theorem \ref{th:tech-strictly} implies that $\Lambda(\hat\phi,\nu)<\log\varrho(L_0^{\hat\phi},L_1^{\hat\phi})$ for every $\nu \in \mathcal{E}_\sigma $ which is not equal to $\mu$. It was noted in \S\ref{se:part2} that the nonempty set $\mathcal{M}_{\max}(\hat\phi)$ is precisely the closed convex hull of $\mathcal{M}_{\max}(\hat\phi)\cap\mathcal{E}_\sigma$, and since the former set is nonempty, so the latter set is nonempty also; but that set contain any elements which are unequal to $\mu$, so we have $\mathcal{M}_{\max}(\hat\phi)\cap\mathcal{E}_\sigma = \{\mu\}$ by elimination and therefore $\mathcal{M}_{\max}(\hat\phi)=\{\mu\}$ for every $(L_0^{\hat\phi},L_1^{\hat\phi})\in\mathcal{U}$. This completes the proof. \end{proof}

Let $\mathcal{M}_{\sigma\times\sigma}$ denote the set of all Borel probability measures on $\Sigma_2\times\Sigma_2$ which are invariant with respect to the transformation $\sigma \times \sigma$. Clearly this set is nonempty. To prove Theorem \ref{th:tech-strictly} we require a lemma:
\begin{lemma}\label{le:dbar-handy}
Let $Z \subseteq \Sigma_2$ be a nonempty compact $\sigma$-invariant set, let $f \colon \Sigma_2 \times \Sigma_2 \to \mathbb{R}$ be continuous, and let $\mu \in \mathcal{E}_\sigma$. Then for $\mu$-a.e. $x \in \Sigma_2$,
\begin{eqnarray*}
\lefteqn{\lim_{n \to \infty}\inf_{y \in Z} \frac{1}{n}\sum_{i=0}^{n-1} f(\sigma^ix,\sigma^iy)}& & \\
& =& \inf\left\{\int f\,dm \colon m \in \mathcal{M}_{\sigma \times \sigma},\,(\pi_1)_*m=\mu,\text{ and }((\pi_2)_*m)(Z)=1\right\}.\end{eqnarray*}
\end{lemma}
\begin{proof}
By applying the subadditive ergodic theorem to the sequence of functions $f_n \colon \Sigma_2 \to \mathbb{R}$ defined by
\[f_n(x):=-\inf\left\{\sum_{i=0}^{n-1}f(\sigma^ix,\sigma^iy) \colon y \in Z\right\},\]
it follows that there exists $\lambda_1 \in \mathbb{R}$ such that for $\mu$-a.e. $x \in \Sigma_2$,
\begin{equation}\label{fear-of-a-spack-planet}\lim_{n \to \infty}\inf_{y \in Z} \frac{1}{n}\sum_{i=0}^{n-1} f(\sigma^ix,\sigma^iy) =\lambda_1.\end{equation}
Let us define
\[\mathcal{J}:=\left\{m \in \mathcal{M}_{\sigma \times \sigma}\colon (\pi_1)_*m=\mu\text{ and }((\pi_2)_*m)(Z)=1\right\}.\]
Since $Z$ is a compact $\sigma$-invariant set, it follows from the Krylov-Bogolioubov theorem that there exists at least one $\nu \in \mathcal{M}_\sigma$ such that $\nu(Z)=1$. In particular $\mu \times \nu \in \mathcal{J}$ and thus $\mathcal{J}$ is nonempty. Let us define
\[\lambda_2:=\inf\left\{\int f\,dm \colon m \in \mathcal{J}\right\}.\]
To prove the lemma we must show that $\lambda_1=\lambda_2$.

We first prove $\lambda_1 \leq \lambda_2$. Let $\varepsilon>0$ and choose $m \in \mathcal{J}$ such that $\int f\,dm <\lambda_2+\varepsilon$. By the Birkhoff ergodic theorem for general invariant measures there exists a bounded measurable function $\overline{f} \colon \Sigma_2 \times \Sigma_2 \to \mathbb{R}$ such that
\[m\left(\left\{(x,y) \in \Sigma_2 \times \Sigma_2 \colon \lim_{n \to \infty}\frac{1}{n}\sum_{i=0}^{n-1}f(\sigma^ix,\sigma^iy) = \overline{f}(x,y)\right\}\right)=1\]
and $\int \overline{f}\,dm=\int f\,dm<\lambda_2+\varepsilon$. Since $m(\Sigma_2 \times Z)=((\pi_2)_*m)(Z)=1$ we therefore have in particular
\[m\left(\left\{(x,y) \in \Sigma_2 \times \Sigma_2 \colon y \in Z\text{ and }\limsup_{n \to \infty}\frac{1}{n}\sum_{i=0}^{n-1}f(\sigma^ix,\sigma^iy) \leq \lambda_2+\varepsilon\right\}\right)>0.\]
On the other hand, since $(\pi_1)_*m = \mu$ we obtain from \eqref{fear-of-a-spack-planet}
\[m\left(\left\{(x,y) \in \Sigma_2 \times \Sigma_2 \colon \lim_{n \to \infty} \frac{1}{n}\inf_{z \in Z}\sum_{i=0}^{n-1} f(\sigma^ix,\sigma^iz)=\lambda_1\right\}\right)=1.\]
It follows that there exist $x \in \Sigma_2$ and $y \in Z$ such that
\[\lambda_1 =  \lim_{n \to \infty} \frac{1}{n}\inf_{z \in Z}\sum_{i=0}^{n-1} f\left(\sigma^ix,\sigma^iz\right)\leq  \limsup_{n \to \infty} \frac{1}{n}\sum_{i=0}^{n-1} f\left(\sigma^ix,\sigma^iy\right) \leq \lambda_2+\varepsilon.\]
Since $\varepsilon>0$ is arbitrary we conclude that $\lambda_1 \leq \lambda_2$ as claimed.

We now claim that there exists $m \in \mathcal{J}$ such that $\int f\,dm = \lambda_1$, which implies that $\lambda_2 \leq \lambda_1$. Let $x \in \Sigma_2$ such that \eqref{fear-of-a-spack-planet} holds and such that additionally $\frac{1}{n}\sum_{i=0}^{n-1}\delta_{\sigma^ix} \to \mu$ in the weak-* topology as $n \to \infty$. Using \eqref{fear-of-a-spack-planet} we may choose a sequence of points $z^{(n)} \in Z$ such that
\begin{equation}\label{two-gentlemen-on-veronal}\lim_{n \to \infty} \frac{1}{n}\sum_{i=0}^{n-1} f\left(\sigma^ix,\sigma^iz^{(n)}\right) =\lambda_1.\end{equation}
Define a sequence $(m_n)_{n=1}^\infty$ of Borel probability measures on $\Sigma_2 \times \Sigma_2$ by
\[m_n:=\frac{1}{n}\sum_{i=0}^{n-1}\delta_{\left(\sigma^ix,\sigma^iz^{(n)}\right)},\]
and choose a strictly increasing sequence of natural numbers $(n_j)_{j=1}^\infty$ and a Borel probability measure $m$ on $\Sigma_2 \times \Sigma_2$ such that $m_{n_j} \to m$ in the limit as $j \to \infty$ with respect to the weak-* topology on the space of Borel probability measures on $\Sigma_2\times\Sigma_2$. If $g \colon \Sigma_2 \times \Sigma_2 \to \mathbb{R}$ is any continuous function, then
\begin{align*}\left|\int g \circ (\sigma \times \sigma)\,dm - \int g\,dm\right| &= \lim_{j \to \infty} \frac{1}{n_j}\left|\int g \circ (\sigma \times \sigma)\,dm_n - \int g\,dm_n\right| \\
&=\lim_{j \to \infty} \left|\frac{1}{n_j}\left(g\left(\sigma^{n_j}x,\sigma^{n_j}z^{(n_j)}\right)-g\left(x,z^{(n_j)}\right)\right)\right|\\
&\leq \lim_{j \to \infty}\frac{2|g|_\infty}{n_j}=0,\end{align*}
and since $g$ is arbitrary it follows that $m$ is $(\sigma \times \sigma)$-invariant. Since $\pi_1$ is continuous we have
\[\left(\pi_1\right)_*m = \lim_{j \to \infty}  \left(\pi_1\right)_*m_{n_j}= \lim_{j \to \infty}\frac{1}{n_j}\sum_{i=0}^{n_j-1}\delta_{\sigma^ix} = \mu\]
by our initial choice of $x$, where the above limits are taken in the weak-* topology on the space of Borel probability measures on $\Sigma_2$. Since $Z$ is closed and $((\pi_2)_*m_n)(Z)=1$ for all $n \geq 1$, we furthermore have
\[((\pi_2)_*m)(Z) \geq \limsup_{j \to \infty} ((\pi_2)_*m_{n_j})(Z)=1\]
and therefore $((\pi_2)_*m)(Z)=1$. Finally, using \eqref{two-gentlemen-on-veronal} we obtain
\[\int f\,dm = \lim_{j \to \infty} \int f\,dm_{n_j}=\lim_{j \to \infty} \frac{1}{n_j}\sum_{i=0}^{n_j-1} f\left(\sigma^ix,\sigma^iz^{(n_j)}\right) =\lambda_1.\]
We conclude that $m \in \mathcal{J}$ with $\int f\,dm=\lambda_1$ and therefore $\lambda_2 \leq \lambda_1$ as claimed. The proof is complete.
\end{proof}
The lemma having been established, we may now give the proof of Theorem \ref{th:tech-strictly}:
\begin{proof}[Proof of Theorem \ref{th:tech-strictly}]
Let $z$ and $Z$ be as in the statement of the theorem, and suppose that $\phi \in \ell_\infty(\mathbb{Z}_2\times\mathbb{Z})$, $\delta>0$ satisfy $\phi(z_i,i)>\phi(1-z_i,i)+\delta$ for every $i \in \mathbb{Z}$. Let $\mu \in\mathcal{E}_\sigma$. Using Lemma \ref{le:dbar-handy} we have for $\mu$-a.e. $x$
\begin{align*}\lim_{n \to \infty}\frac{1}{n} \inf_{y \in Z}\sum_{i=0}^{n-1}\mathfrak{d} (\sigma^ix,\sigma^{i}y)&=\inf\left\{\int \mathfrak{d}\,dm \colon m \in \bigcup_{\nu \in \mathcal{M}_\sigma \colon \nu(Z)=1} J(\mu,\nu)\right\}\\
&=\inf\{\overline{d}(\mu,\nu) \colon \nu \in\mathcal{M}_\sigma\text{ and } \nu(Z)=1\}.\end{align*}
For each $x \in \Sigma_2$, $n \geq 1$ and $k \in \mathbb{Z}$ we may estimate
\[\sum_{i=0}^{n-1}\phi(x_i,k+i)-\phi(z_{k+i},k+i) \leq -\delta \sum_{i=0}^{n-1}\mathfrak{d}(\sigma^ix,\sigma^{k+i}z)\] 
by virtue of the hypothesis on $\phi$. If $n \geq 1$ and $x \in \Sigma_2$ then
\[\log \left\|L_{x_n}^\phi \cdots L_{x_1}^\phi\right\| = \sup_{k \in \mathbb{Z}} \sum_{i=0}^{n-1}\phi(x_i,k+i)\]
and therefore
\begin{align*}\log \left\|L_{x_n}^\phi \cdots L_{x_1}^\phi\right\|&= \sup_{k \in \mathbb{Z}}\left(\sum_{i=0}^{n-1}\phi(z_{k+i},k+i) + \sum_{i=0}^{n-1}\left(\phi(x_i,k+i)-\phi(z_{k+i},k+i)\right)\right)\\
&\leq \sup_{k \in \mathbb{Z}}\sum_{i=0}^{n-1}\phi(z_{k+i},k+i) - \inf_{k \in \mathbb{Z}}\delta \sum_{i=0}^{n-1}\mathfrak{d}(\sigma^ix,\sigma^{k+i}z)\\
&\leq \sup_{y \in \Sigma_2} \log\|L^\phi_{y_n}\cdots L^\phi_{y_1}\| - \inf_{y \in Z}\delta \sum_{i=0}^{n-1}\mathfrak{d}(\sigma^ix,\sigma^iy).\end{align*}
It follows via the subadditive ergodic theorem that for $\mu$-a.e. $x \in \Sigma_2$
\begin{align*}\Lambda(\phi,\mu)&=
\lim_{n \to \infty}\frac{1}{n} \log\left\|L_{x_n}\cdots L_{x_1}\right\|\\
&=\lim_{n \to \infty}\frac{1}{n}\sup_{k \in \mathbb{Z}}\sum_{i=0}^{n-1}\phi(x_i,k+i)\\
&\leq \lim_{n \to \infty}\sup_{y \in \Sigma_2} \frac{1}{n}\log\|L^\phi_{y_n}\cdots L^\phi_{y_1}\| - \lim_{n \to \infty}\frac{\delta}{n}\inf_{y \in Z}\sum_{i=0}^{n-1}\mathfrak{d} (\sigma^ix,\sigma^{i}y)\\
&=\log\varrho\left(L_0^{\phi},L_1^{\phi}\right)-\delta \inf\left\{\overline{d}(\mu,\nu) \colon \nu\in\mathcal{M}_\sigma \text{ and }\nu(Z)=1\right\}\end{align*}
as required. 

Let us now consider the special case $\phi(z_i,i)=1$, $\phi(1-z_i,i)=0$ for all $i \in \mathbb{Z}$. Let $\mu \in \mathcal{E}_\sigma$. By Lemma \ref{le:dbar-handy} we again obtain
\[\lim_{n \to \infty} \frac{1}{n}\inf_{y \in Z}\sum_{i=0}^{n-1}\mathfrak{d} (\sigma^ix,\sigma^{i}y) = \inf\{\overline{d}(\mu,\nu) \colon \nu(Z)=1\}\]
for $\mu$-a.e. $x \in \Sigma_2$, but the stronger hypothesis on $\phi$ in this case yields
\[\sum_{i=0}^{n-1}\phi(x_i,k+i)=  \sum_{i=0}^{n-1}\left(\phi(z_{k+i},k+i)-\mathfrak{d}(\sigma^ix,\sigma^{k+i}z)\right)=n-\sum_{i=0}^{n-1}\mathfrak{d}(\sigma^ix,\sigma^{k+i}z)\] 
for every $x \in \Sigma_2$, $k \in \mathbb{Z}$ and $n \geq 1$. Hence, for each $x \in \Sigma_2$ and $n \geq 1$
\[\frac{1}{n}\log\left\|L^\phi_{x_n}\cdots L^\phi_{x_1}\right\| = \sup_{k \in \mathbb{Z}}\left(1-\frac{1}{n}\sum_{i=0}^{n-1}\mathfrak{d}(\sigma^ix,\sigma^{k+i}z)\right)=1-\frac{1}{n}\inf_{y \in Z}\sum_{i=0}^{n-1}\mathfrak{d}(\sigma^ix,\sigma^iy),\]
and it follows via the subadditive ergodic theorem that for $\mu$-a.e. $x \in \Sigma_2$
\begin{align*}\Lambda(\phi,\mu)&=\lim_{n \to \infty}\frac{1}{n}\log\left\|L^\phi_{x_n}\cdots L^\phi_{x_1}\right\|\\
 &= 1-\lim_{n \to \infty}\frac{1}{n}\inf_{y \in Z}\sum_{i=0}^{n-1}\mathfrak{d}(\sigma^ix,\sigma^{i}y)\\
&=1-\inf\left\{\overline{d}(\mu,\nu)\colon \nu \in \mathcal{E}_\sigma\text{ and }\nu(Z)=1\right\}\end{align*}
as required.
\end{proof}

%%%%%%%%%%%%%%%%%%%%%%%%%%%
%                  NEW SECTION
%%%%%%%%%%%%%%%%%%%%%%%%%%%

\section{Proof of Theorem \ref{th:nomather}}
Before starting the proof we note the following simple result:
\begin{lemma}\label{le:lemino}
Let $Z \subseteq \Sigma_2$ be closed and $\sigma$-invariant, and let $f \colon  \Sigma_2 \to \mathbb{R}$ be continuous. Then
\[\lim_{n \to \infty}\sup_{x \in Z}\frac{1}{n}\sum_{i=0}^{n-1}f(\sigma^ix) = \sup_{\substack{\mu \in\mathcal{M}_\sigma\\\mu(Z)=1}} \int f\,d\mu.\]
\end{lemma}
\begin{proof}
Apply Theorem \ref{th:stst} with $X:=Z$, $f_n:=\sum_{i=0}^{n-1}f(\sigma^ix)$.
\end{proof}

\begin{proof}[Proof of Theorem \ref{th:nomather}]
A classical construction due to J. Oxtoby  \cite{Ox52} gives an explicit example of a sequence $z \in \Sigma_2$ such that $Z:=\overline{\{\sigma^i z \colon i \in\mathbb{Z}\}}$ is minimal but not uniquely ergodic. If we define $[1]:=\{x \in \Sigma_2 \colon x_0=1\}$ then Oxtoby's example has the specific property that there exist two ergodic measures $\mu_1, \mu_2$ such that $\mu_1(Z)=\mu_2(Z)=1$, but $\mu_1([1])>\mu_2([1])$. S. Williams \cite{Wi84} subsequently showed that there exist \emph{exactly} two ergodic measures supported on Oxtoby's minimal set $Z$, and we will make use of this fact below, constructing a pair $(L_0^\phi,L_1^\phi)$ for which $\mu_1$ is maximising and $\mu_2$ is not. Our argument can easily be adapted to minimal sets $Z$ in which a greater number of ergodic measures exist, in which case our argument shows that only those measures $\mu$ supported on $Z$ which maximise $\mu([1])$ can be maximising measures for the associated pair $(L_0^\phi,L_1^\phi)$.

So, let $z$, $Z$ and $\mu_1,\mu_2$ be as described above. If we define $\phi \in \ell_\infty(\mathbb{Z}_2\times\mathbb{Z})$ by
\[\phi(a,i):=\left\{\begin{array}{cl}2&\text{if }a=z_i=1\\
1&\text{if }a=z_i=0\\
0&\text{if }a \neq z_i
\end{array}\right.\]
then  $\mathcal{M}_{\max}(\phi)\cap\mathcal{E}_\sigma \subseteq \{\mu_1,\mu_2\}$ by force of Theorem \ref{th:tech-strictly}. For all $x \in \Sigma_2$, $n \geq 1$ and $k \in \mathbb{Z}$ we have
\begin{align*}\sum_{i=0}^{n-1}\phi(x_i,i+k)& = \sum_{i=0}^{n-1}\phi(z_{i+k},i+k) + \sum_{i=0}^{n-1} \left(\phi(x_i,i+k)-\phi(z_{i+k},i+k)\right)\\
&= n +\sum_{i=0}^{n-1} \chi_{[1]}(\sigma^{i+k}z)  - \sum_{i=0}^{n-1}\mathfrak{d}(\sigma^ix,\sigma^{i+k}z),\end{align*}
so taking the supremum with respect to $k \in \mathbb{Z}$ it follows that for each $x \in \Sigma_2$ and $n \geq 1$
\[\frac{1}{n}\log\left\|L^\phi_{x_n}\cdots L^\phi_{x_1}\right\| = 1-\inf_{y \in Z}\frac{1}{n}\sum_{i=0}^{n-1}\left( \mathfrak{d}(\sigma^ix,\sigma^iy)-\chi_{[1]}(\sigma^iy)\right).\]
Apply Lemma \ref{le:dbar-handy} with $f(x,y):=\mathfrak{d}(x,y)-\chi_{[1]}(y)$ we find that for $\mu_2$-a.e. $x$
\[\Lambda(\phi,\mu_2)=\lim_{n \to\infty}\frac{1}{n}\log\left\|L^\phi_{x_n}\cdots L^\phi_{x_1}\right\|=1-\inf\left\{\int f\,dm\colon  m \in \bigcup_{\substack{\mu \in \mathcal{M}_\sigma\\\mu(Z)=1}} \mathcal{J}(\mu_2,\mu)\right\}.\]
Using the weak-* compactness of $\mathcal{M}_{\sigma \times \sigma}$ one may easily show that there exists an $m$ which attains this infimum. The measure $\mu:=(\pi_2)_*m\in\mathcal{M}_\sigma$ satisfies $\mu(Z)=1$ and is therefore equal to $\beta\mu_1+(1-\beta)\mu_2$ for some real number $\beta \in [0,1]$. If $\beta \neq 0$ then
\[\int f\,dm = \int \mathfrak{d}\,dm - \mu([1]) \geq \overline{d}(\mu_2,\mu)-\mu([1])>-\mu([1])\geq -\mu_1([1])\]
and if $\beta=0$ then
\[\int f\,dm = \int \mathfrak{d}\,dm - \mu_2([1]) \geq -\mu_2([1])>-\mu_1([1])\]
so we conclude that
\[\Lambda(\phi,\mu_2)<1+\mu_1([1]).\]
Applying Lemma \ref{le:lemino} we observe that since $\{\sigma^n z \colon n \in\mathbb{Z}\}$ is dense in $Z$, 
\[ \lim_{n \to \infty} \frac{1}{n}\sup_{k \in\mathbb{Z}}\sum_{i=0}^{n-1}\chi_{[1]}(\sigma^{i+k}z)=\lim_{n \to \infty}\sup_{x \in Z}\sum_{i=0}^{n-1}\chi_{[1]}(\sigma^ix) = \sup_{\substack{\nu\in\mathcal{M}_\sigma\\\nu(Z)=1}}\nu([1])=\mu_1([1]).\]
Since therefore
\begin{align*}\log\varrho\left(L_0^\phi,L_1^\phi\right)&\geq \liminf_{n \to \infty}\frac{1}{n}\log\left\|L^\phi_{z_n}\cdots L^\phi_{z_1}\right\|\\
&= \liminf_{n \to \infty} \frac{1}{n}\sup_{k \in\mathbb{Z}}\sum_{i=0}^{n-1}\phi(z_i,i+k)\\
&=\liminf_{n \to\infty}\left(1+\frac{1}{n}\sup_{k \in \mathbb{Z}}\sum_{i=0}^{n-1}\chi_{[1]}(\sigma^{i+k}z) \right)=1+\mu_1([1])\end{align*}
we conclude that $\Lambda(\phi,\mu_2)<\log\varrho(L_0^\phi,L_1^\phi)$, so that $\mu_2$ is not maximising. It follows that $\mathcal{M}_{\max}(\phi)\cap\mathcal{E}_\sigma=\{\mu_1\}$, and since $\mathcal{M}_{\max}(\phi)$ is the closed convex hull of that set we necessarily have $\mathcal{M}_{\max}(\phi)=\{\mu_1\}$. Since $Z$ is minimal, $\mu_1$ and $\mu_2$ both have support equal to $Z$. We conclude that $(L_0^\phi,L_1^\phi)$ has the properties claimed in the statement of Theorem \ref{th:nomather}.\end{proof}

%%%%%%%%%%%%%%%%%%%%%%%%%%%
%                  NEW SECTION
%%%%%%%%%%%%%%%%%%%%%%%%%%%

\section{Proof of Theorem \ref{th:unique}}
In this section we prove the following result which is easily seen to imply Theorem \ref{th:unique}:
\begin{theorem}\label{pieless-in-gaza}
The set
\[\left\{\phi \in \ell_\infty(\mathbb{Z}_2\times\mathbb{Z}) \colon \left(L_0^{\phi},L_1^{\phi}\right)\text{ has a unique maximising measure}\right\}\]
is a dense $G_\delta$ subset of $\ell_\infty(\mathbb{Z}_2\times\mathbb{Z})$.
\end{theorem}
Before beginning the proof let us fix some notation. We shall say that a \emph{word} of length $N$ is an ordered sequence of $N$ symbols each of which is either zero or one. We will use the symbol $\omega$ to denote an arbitrary word. Given a word $\omega$ of length $N$ we will use the symbols $\omega_0,\omega_1,\ldots,\omega_{N-1}$ to denote its successive symbols, so that for example if $\omega$ is the word $10010$ then $\omega_0=\omega_3=1$ and $\omega_1=\omega_2=\omega_4=0$. Given a word $\omega$ of length $N$ we define the \emph{cylinder set} associated to $\omega$ to be the set
\[[\omega]:=\left\{(x_i)_{i \in \mathbb{Z}} \in \Sigma_2 \colon x_i=\omega_i\text{ for all }i=0,\ldots,N-1\right\}.\]
It is well-known that shift-invariant measures are completely determined by the values which they give to cylinder sets: if two measures $\mu_1,\mu_2 \in \mathcal{M}_\sigma$ are distinct from one another then there exists a word $\omega$ such that $\mu_1([\omega]) \neq \mu_2([\omega])$. We observe that for any word $\omega$ the set $[\omega]$ is clopen and hence the function $\chi_{[\omega]}$ is continuous.

The following lemma is of a standard type; for similar arguments in different contexts see e.g. \cite[Lemma 3.1]{Mo10c} and the proofs of \cite[Proposition 10]{CoLoTh01} and \cite[Theorem 3.2]{Je06a}.
\begin{lemma}\label{gfs}
For every finite word $\omega$ and real number $\varepsilon>0$, the set
\[\mathcal{O}_{\omega,\varepsilon}:=\left\{\phi \in \ell_\infty(\mathbb{Z}_2 \times \mathbb{Z}) \colon |\mu_1([\omega])-\mu_2([\omega])|<\varepsilon\text{ for all }\mu_1,\mu_2 \in \mathcal{M}_{\max}(\phi)\right\}\]is open.
\end{lemma}
\begin{proof}
Let $\omega$ and $\varepsilon$ be as in the statement of the lemma. We will show that the set
\[\mathcal{D}_{\omega,\varepsilon}:=\left\{\phi \in \ell_\infty(\mathbb{Z}_2 \times \mathbb{Z}) \colon \exists \,\mu_1,\mu_2 \in \mathcal{M}_{\max}(\phi)\text{ such that }|\mu_1([\omega])-\mu_2([\omega])| \geq\varepsilon\right\}\]
is closed. To this end let us suppose that $(\phi_n)$ is a sequence of elements of $\mathcal{D}_{\omega,\varepsilon}$ which converges to a limit $\phi \in \ell_\infty(\mathbb{Z}_2 \times \mathbb{Z})$. For $i=1,2$ let $(\mu_{i,n})_{n=1}^\infty$ be a sequence of measures such that $\mu_{i,n} \in \mathcal{M}_{\max}(\phi_n)$ and $|\mu_{1,n}([\omega])-\mu_{2,n}([\omega])| \geq \varepsilon$ for all $n \geq 1$. Since $\mathcal{M}_\sigma$ is weak-* compact we may choose a strictly increasing sequence of integers $(n_j)_{j=1}^\infty$ and two measures $\mu_1,\mu_2 \in \mathcal{M}_\sigma$ such that $\lim_{j \to \infty} \mu_{i,n_j}=\mu_i$ in the weak-* topology for $i=1,2$. By Lemma \ref{kuratowski} we have $\mu_1,\mu_2 \in \mathcal{M}_{\max}(\phi)$. Since $[\omega]\subset \Sigma_2$ is clopen the function $\nu \mapsto \nu([\omega])$ is a continuous function from $\mathcal{M}_\sigma$ to $[0,1]$, and follows that $|\mu_1([\omega])-\mu_2([\omega])| \geq \varepsilon$. We conclude that $\phi \in \mathcal{D}_{\omega,\varepsilon}$ and therefore $\mathcal{D}_{\omega,\varepsilon}$ is closed as required. The proof is complete.
\end{proof}
We also require the following more technical lemma, the proof of which, due to its length, we postpone until the following section:
\begin{lemma}\label{xyz}
Let $\phi \in \ell_\infty(\mathbb{Z}_2 \times \mathbb{Z})$, let $\omega$ be a finite word of length $N$, and let $\mu \in \mathcal{M}_{\max}(\phi)$.
Then there exists $\psi \in \ell_\infty(\mathbb{Z}_2 \times \mathbb{Z})$ such that:
\begin{enumerate}
\item
For each $\nu \in \mathcal{M}_\sigma$, $\Lambda(\psi,\nu) \leq N\nu([\omega])$.
\item
For each $\lambda>0$, $\log\varrho\left(L_0^{\phi+\lambda\psi},L_1^{\phi+\lambda\psi}\right)\geq \log\varrho\left(L_0^{\phi},L_1^{\phi}\right)+N\lambda\mu([\omega])$.
\end{enumerate}
\end{lemma}
We may now present the proof the theorem stated at the beginning of the section. Our strategy is influenced by the proof of \cite[Theorem 3.2]{Je06a}.

\begin{proof}[Proof of Theorem \ref{pieless-in-gaza}.]
For each finite word $\omega$ and real number $\varepsilon>0$, let
\[\mathcal{O}_{\omega,\varepsilon}:=\left\{\phi \in \ell_\infty(\mathbb{Z}_2 \times \mathbb{Z}) \colon |\mu_1([\omega])-\mu_2([\omega])|<\varepsilon\text{ for all }\mu_1,\mu_2 \in \mathcal{M}_{\max}(\phi)\right\}.\]
Since invariant measures are completely determined by their values on cylinder sets, we have
\[\left\{\phi \in \ell_\infty(\mathbb{Z}_2 \times \mathbb{Z}) \colon \mathcal{M}_{\max}(\phi)\text{ is a singleton}\right\}=\bigcap_{\omega}\bigcap_{n=1}^\infty \mathcal{O}_{\omega,2^{-n}}.\]
Hence to prove the theorem it is sufficient to show that each of the sets $\mathcal{O}_{\omega,\varepsilon}$ is open and dense in $\ell_\infty(\mathbb{Z}_2\times\mathbb{Z})$. Since each $\mathcal{O}_{\omega,\varepsilon}$ is open by Lemma \ref{gfs}, we have only to show that each of these sets is also dense.

Fix $\varepsilon>0$ and a word $\omega$ of length $N$, and let $\phi \in \ell_\infty(\mathbb{Z}_2 \times \mathbb{Z})$. Since $\mathcal{M}_{\max}(\phi)$ is compact we may choose $\mu \in \mathcal{M}_{\max}(\phi)$ such that $\mu([\omega]) \geq \nu([\omega])$ for all $\nu \in \mathcal{M}_{\max}(\phi)$. Let $\psi \in \ell_\infty(\mathbb{Z}_2 \times \mathbb{Z})$ be as given by Lemma \ref{xyz} in respect of the function $\phi$ and measure $\mu$. For each $n \geq 1$ define $\phi_n:=\phi+2^{-n}\psi$. We will show that $\phi_n \in \mathcal{O}_{\omega,\varepsilon}$ for all sufficiently large $n$.

 We claim that if $(\nu_n)_{n=1}^\infty$ is a sequence of measures such that $\nu_n \in \mathcal{M}_{\max}(\phi_n)$ for each $n \geq 1$, then $\nu_n([\omega]) \to \mu([\omega])$ in the limit as $n \to \infty$. To prove the claim we first observe that by Lemma \ref{xyz} and Lemma \ref{le:simples} we have for every $n \geq 1$
\begin{align*}\log\varrho\left(L_0^{\phi},L_1^{\phi}\right)+2^{-n}N\mu([\omega]) &\leq \log\varrho\left(L_0^{\phi_{n}},L_1^{\phi_n}\right)\\
&=\Lambda(\phi_{n},\nu_{n}) \\
&\leq \Lambda(\phi,\nu_{n}) +\Lambda(2^{-n}\psi,\nu_{n})\\
& \leq \log\varrho\left(L_0^\phi,L_1^\phi\right) + 2^{-n}N\nu_{n}([\omega]),\end{align*}
and therefore $\nu_n([\omega]) \geq \mu([\omega])$ for all $n$. In particular $\liminf_{n \to \infty} \nu_n([\omega]) \geq \mu([\omega])$. Now choose a strictly increasing sequence of natural numbers $(n_j)_{j=1}^\infty$ such that $\lim_{j \to \infty} \nu_{n_j}([\omega])=\limsup_{n \to \infty}\nu_n([\omega])$. Since each $\nu_{n_j}$ belongs to $\mathcal{M}_{\max}(\phi_{n_j})$, using Lemma \ref{kuratowski} we may find a finer subsequence $(n_j)$ and measure $\nu \in \mathcal{M}_{\max}(\phi)$ such that $\nu_{n_j} \to \nu$ in the weak-* topology as $j \to \infty$. By our choice of $\mu$ we have $\nu([\omega]) \leq \mu([\omega])$, and therefore
\[\limsup_{n \to \infty} \nu_n([\omega])=\lim_{j \to \infty}\nu_{n_j}([\omega]) = \nu([\omega]) \leq \mu([\omega])\]
which proves the claim. Using the claim it follows easily that
\[\lim_{n \to \infty} \sup\left\{\nu([\omega]) \colon \nu \in \mathcal{M}_{\max}(\phi_n)\right\} = \lim_{n \to \infty} \inf\left\{\nu([\omega]) \colon \nu \in \mathcal{M}_{\max}(\phi_n)\right\} =\mu([\omega])\]
and hence for all sufficiently large $n$ we have $\phi_n \in \mathcal{O}_{\omega,\varepsilon}$ as required. Since $\phi_n \to \phi$ in the limit as $n \to \infty$ we deduce that $\phi \in\overline{\mathcal{O}_{\omega,\varepsilon}}$, and since $\phi$ is arbitrary it follows that $\mathcal{O}_{\omega,\varepsilon}$ is dense. This completes the proof of the theorem.
\end{proof}

\section{Proof of Lemma \ref{xyz}}\label{se:lemmer}

 Let $\omega$ be a word of length $N$, and let $\mu \in\mathcal{M}_{\max}(\phi)$. If $\mu([\omega])=0$ then we may simply take $\psi \equiv 0$, so we shall assume for the remainder of the proof that $\mu([\omega])>0$.

The function $\psi \in \ell_\infty(\mathbb{Z}_2 \times \mathbb{Z})$ will be drawn from a class of functions defined as follows. Let $A \subseteq \mathbb{Z}$ be an arbitrary subset of the integers. For each $\ell \in A$ define a function $\psi_\ell \in \ell_\infty(\mathbb{Z}_2 \times \mathbb{Z})$ by $\psi_\ell(\omega_{i},\ell+i)=1$ and $\psi_\ell(1-\omega_{i},\ell+i)=-N$ for all $i=0,\ldots,N-1$, and $\psi_\ell(a,k)=0$ otherwise. Clearly each $\psi_\ell$ satisfies $-N \leq \psi_\ell \leq 1$, and the estimate
\[-N^2 \leq \sum_{i=a}^b \psi_\ell(x_i,k+i) \leq N\]
is valid for all $x \in \Sigma_2$, $k \in \mathbb{Z}$ and $a \leq b$. Similarly, for every $(a,k) \in \mathbb{Z}_2 \times \mathbb{Z}$ we have $\psi_\ell(a,k) \neq 0$ for at most $N$ different values of $\ell\in A$. We may therefore define the function $\psi_A \in \ell_\infty(\mathbb{Z}_2 \times \mathbb{Z})$ associated to a set $A\subseteq \mathbb{Z}$ by $\psi_A:=\sum_{\ell \in A} \psi_\ell$, and the preceding observations imply that this series converges pointwise and satisfies $-N^2 \leq \psi_A \leq N$.

We first show that the function $\psi_A$ satisfies (i) irrespective of the precise choice of the set of integers $A$. Fix an arbitrary set $A\subseteq \mathbb{Z}$. In order to prove our assertion we claim that for every $x \in \Sigma_2$, $k \in \mathbb{Z}$ and $n >N$,
\begin{equation}\label{hippiedharma}\sum_{i=0}^{n-1} \psi_A(x_i,k+i) \leq 2N(N-1) + N \sum_{i=0}^{n-1}\chi_{[\omega]}(\sigma^ix).\end{equation}

Let us fix $x \in \Sigma_2$, $k \in \mathbb{Z}$ and $n > N$, and prove the claim. If $\ell \leq k-N$ then $\psi_\ell(x_i,k+i)$ is zero for all $i=0,\ldots,n-1$, since $k+i \geq k\geq \ell+N$ for all such $i$. Similarly, if $\ell \geq n+k$ then $\psi_\ell(x_i,k+i)=0$ for all $i=0,\ldots,n-1$ since $k+i \leq k+n -1< \ell$. We therefore have
\[\sum_{i=0}^{n-1}\psi_A(x_i,k+i)=\sum_{i=0}^{n-1}\sum_{\ell \in A}\psi_\ell(x_i,k+i) = \sum_{\ell \in A \cap [k-N+1,n+k-1]}\sum_{i=0}^{n-1} \psi_\ell(x_i,k+i).\]
For each $\ell \in A$ the quantity $\psi_\ell(x_i,k+i)$ can be nonzero for at most $N$ different values of $i \in \mathbb{Z}$, so we have $\sum_{i=0}^{n-1}\psi_\ell(x_i,k+i) \leq N \sup \psi_\ell =N$ for every $\ell \in A$. It follows that
\begin{equation}\label{indianmafiainaction}\sum_{i=0}^{n-1}\psi_A(x_i,k+i) \leq 2N(N-1) + \sum_{\ell \in A \cap [k,n+k-N]}\sum_{i=0}^{n-1} \psi_\ell(x_i,k+i).\end{equation}
If $\ell \in A \cap [k,n+k-N]$, then since $0 \leq \ell-k \leq \ell-k+N -1\leq n-1$ we have
\begin{equation}\label{lagerlagerlager}\sum_{i=0}^{n-1} \psi_\ell(x_i,k+i)=\sum_{i=\ell-k}^{\ell-k+N-1} \psi_\ell(x_i,k+i)=\sum_{j=0}^{N-1}\psi_\ell(x_{\ell-k+j},\ell+j).\end{equation}
Now, if $\sigma^{\ell-k}x \in [\omega]$ then $\omega_{j}=x_{\ell-k+j}$ for every $j=0,\ldots,N-1$ and the sum \eqref{lagerlagerlager} is equal to $N$ by the definition of $\psi_\ell$. On the other hand, if $\sigma^{\ell-k}x \notin [\omega]$ then for at least one such $j$ we have $x_{\ell-k+j}\neq \omega_j$ and therefore $\psi_\ell(x_{\ell-k+j},\ell+j) =-N$. For each other $j$ in the range $0 \leq j \leq N-1$ we clearly have $\psi_\ell(x_{\ell-k+j},\ell+j) \leq 1$, and it follows that when $\sigma^{\ell-k}x \notin [\omega]$ the sum \eqref{lagerlagerlager} is negative. We have shown in particular that for every $\ell \in A \cap [k,n+k-N]$
\[\sum_{i=0}^{n-1} \psi_\ell(x_i,k+i) \leq N\chi_{[\omega]}(\sigma^{\ell-k}x)\]
and it follows that
\begin{align*}\sum_{\ell \in A \cap [k,n+k-N]}\sum_{i=0}^{n-1} \psi_\ell(x_i,k+i) &\leq \sum_{\ell \in A \cap [k,n+k-N]}N\chi_{[\omega]}(\sigma^{\ell-k}x)\\
&\leq N\sum_{\ell =k}^{n+k-N}\chi_{[\omega]}(\sigma^{\ell-k}x)\\
&= N\sum_{i=0}^{n-N}\chi_{[\omega]}(\sigma^{i}x),\end{align*}
which combined with \eqref{indianmafiainaction} implies the claimed inequality \eqref{hippiedharma}.

To see that the validity of \eqref{hippiedharma} implies the validity of (i) for the function $\psi:=\psi_A$, irrespective of the choice of $A \subseteq \mathbb{Z}$, we argue as follows. Let $\nu \in \mathcal{M}_\sigma$ be any invariant measure. For each $n >N$ and $x \in \Sigma_2$ we have
\[\sup_{k \in \mathbb{Z}} \sum_{i=0}^{n-1}\psi_A(x_i,k+i) \leq 2N(N-1)+N\sum_{i=0}^{n-1}\chi_{[\omega]}(\sigma^ix)\]
by \eqref{hippiedharma}, and so by integration we have for every $n >N$
\[\frac{1}{n}\int \left(\sup_{k \in \mathbb{Z}}\sum_{i=0}^{n-1}\psi_A(x_i,k+i)\right)d\nu(x) \leq \frac{2N(N-1)}{n} + N\nu([\omega]).\]
Taking the limit $n \to \infty$ clearly yields (i) as claimed. We conclude that $\psi_A$ satisfies (i) for every choice of the set $A\subseteq \mathbb{Z}$.

We shall now prove that there exists a set $A \subseteq \mathbb{Z}$ such that the function $\psi_A$ also satisfies (ii). By the subadditive ergodic theorem for general invariant measures there exists a measurable function $\Phi \colon \Sigma_2 \to \mathbb{R}$ such that
\[\mu\left(\left\{x \in \Sigma_2 \colon \lim_{n \to \infty} \frac{1}{n}\sup_{k \in \mathbb{Z}}\sum_{i=0}^{n-1} \phi(x_i,k+i) = \Phi(x)\right\}\right)=1\]
and
\[\int \Phi\,d\mu = \lim_{n \to \infty} \frac{1}{n}\int \left(\sup_{k \in \mathbb{Z}} \sum_{i=0}^{n-1}\phi(x_i,k+i)\right)d\mu(x) =\Lambda(\phi,\mu)=\log\varrho\left(L_0^\phi,L_1^\phi\right).\]
It follows from the definition of the joint spectral radius that
\[\mu\left(\left\{x \in \Sigma_2 \colon \Phi(x)>\log\varrho\left(L_0^\phi,L_1^\phi\right)\right\}\right)=0,\]
and since $\int \Phi\,d\mu = \log\varrho(L_0^\phi,L_1^\phi)$ we deduce that $\Phi(x)=\log\varrho(L_0^\phi,L_1^\phi)$ $\mu$-a.e. By the Birkhoff ergodic theorem for general invariant measures there exists $\overline{\chi} \colon \Sigma_2 \to \mathbb{R}$ such that $\int \overline{\chi}\,d\mu=\mu([\omega])$ and
\[ \mu\left(\left\{\lim_{n \to \infty} \frac{1}{n}\sum_{i=0}^{n-1} \chi_{[\omega]}(\sigma^i x)=\overline{\chi}(x)\right\}\right)=1.\]
Combining these results we deduce that there exists $z \in \Sigma_2$ such that
 \begin{equation}\label{eq:z-good-for-chi}\lim_{n \to \infty}\frac{1}{n}\sum_{i=0}^{n-1}\chi_{[\omega]}(\sigma^iz)=\overline{\chi}(z) \geq \mu([\omega])\end{equation}
and
\begin{equation}\label{eq:z-good-for-phi}\lim_{n \to \infty}\frac{1}{n} \sup_{k \in \mathbb{Z}}\sum_{i=0}^{n-1}\phi(z_i,k+i)=\log\varrho\left(L_0^\phi,L_1^\phi\right).\end{equation}
Let us fix such a $z$ for the remainder of the proof. Define a strictly increasing sequence of natural numbers $(n_j)_{j=1}^\infty$ as follows: let $n_1:=N$ and, the integers $n_1,\ldots,n_{j-1}$ having been defined, choose $n_j$ to be the smallest natural number such that
\begin{equation}\label{eq:enjay}\frac{1}{n_j}\sum_{i=1}^{j-1} n_i<\frac{1}{2^{j+1}}.\end{equation}
In view of \eqref{eq:z-good-for-phi} we may choose a sequence of integers $(k_j)_{j=1}^\infty$ such that
\begin{equation}\label{eq:k-good-for-phi}\lim_{j \to \infty} \frac{1}{n_j}\sum_{i=0}^{n_j-1}\phi(z_i,k_j+i)=\log\varrho\left(L_0^\phi,L_1^\phi\right).\end{equation}
Let us define a sequence of sets $B_j\subseteq \mathbb{Z}$ inductively by taking $B_0:=\emptyset$, and $B_j := \{k_j,\ldots,k_j+n_j-N\} \setminus \bigcup_{i=0}^{j-1}B_i$ for each $j \geq 1$ once the sets $B_0,\ldots,B_{j-1}$ have been defined. Now define
\[A_j:=\left\{\ell \in B_j \colon \sigma^{\ell-k_j}z \in [\omega]\right\}\]
for each $j \geq 0$, and
\[A:=\bigcup_{j=0}^\infty A_j.\]
We observe that the sets $B_j$ are pairwise disjoint, and since $A_j \subseteq B_j$ for each $j \geq 0$ the sets $A_j$ are pairwise disjoint also. By construction we have
\[\{k_j,\ldots,k_j+n_j-N\} \subseteq \bigcup_{i=0}^j B_j\]
and using pairwise disjointness this immediately implies 
\[[k_j,k_j+n_j-N] \cap \bigcup_{i=j+1}^\infty B_i = \emptyset\]
for every $j \geq 1$. This last expression in turn implies
\begin{equation}\label{eq:a-decomposition}A \cap [k_j,k_j+n_j-N] = \bigcup_{i=0}^j A_i \cap [k_j,k_j+n_j-N] = A_j \cup \left(\bigcup_{i=0}^{j-1} A_i \cap [k_j,k_j+n_j-N]\right)\end{equation}
for every $j \geq 1$. Finally we define
\[C_j:=\left\{\ell \in \mathbb{Z} \colon k_j \leq \ell \leq k_j + n_j -N\text{ and } \sigma^{\ell-k_j}z \in [\omega]\right\}\] 
for every $j \geq 1$. We observe that $A_j=  C_j \setminus \bigcup_{i=0}^{j-1}B_i$ for every $j \geq 1$.

By definition we have $\psi_\ell(a,k)=0$ when $k < \ell$ and when $k \geq\ell+N$, so if $k_j+n_j\leq \ell$ or $\ell \leq k_j-N$ then $\sum_{i=0}^{n_j-1}\psi_\ell(z_i,k_j+i)=0$ since in this case every summand is zero. Thus
\[\sum_{i=0}^{n_j-1}\psi_A(z_i,k_j+i)=\sum_{i=0}^{n_j-1}\sum_{\ell \in A} \psi_\ell(z_i,k_j+i) = \sum_{i=0}^{n_j-1}\sum_{\ell \in A \cap [k_j-N+1,k_j+n_j-1]} \psi_\ell(z_i,k_j+i)\]
and therefore
\begin{eqnarray*}\lefteqn{\left|\sum_{i=0}^{n_j-1}\sum_{\ell \in A} \psi_\ell(z_i,k_j+i)-\sum_{i=0}^{n_j-1}\sum_{\ell \in A \cap [k_j,k_j+n_j-N]} \psi_\ell(x_i,k_j+i)\right|}\\
& &\leq \sum_{\ell \in A \cap \left([k_j-N+1,k_j-1]\cup[k_j+n_j-N+1,k_j+n_j-1]\right)}\left|\sum_{i=0}^{n_j-1} \psi_\ell(z_i,k_j+i)\right|\leq 2N^2(N-1)\end{eqnarray*}
where we have used the fact that $|\sum_{i =0}^{n_j-1}\psi_\ell(x_i,k_j+i)| \leq N^2$ for every $\ell \in A$. Now, using \eqref{eq:a-decomposition}, the fact that the sets $A_r$ are pairwise disjoint and finally the inequality \eqref{eq:enjay} we find that the difference 
\[\left|\sum_{i=0}^{n_j-1}\sum_{\ell \in A \cap [k_j,k_j+n_j-N]} \psi_\ell(z_i,k_j+i) - \sum_{i=0}^{n_j-1}\sum_{\ell \in A_j} \psi_\ell(z_i,k_j+i)\right|\]
is bounded by
\begin{align*}\left|\sum_{i=0}^{n_j-1}\sum_{r=0}^{j-1}\sum_{\ell \in A_r \cap [k_j,k_j+n_j-N]}\psi_\ell(z_i,k_j+i)\right|&\leq\sum_{r=0}^{j-1}\sum_{\ell \in A_r}\left|\sum_{i=0}^{n_j-1}\psi_\ell(z_i,k_j+i)\right|\\
&\leq \sum_{r=0}^{j-1}\sum_{\ell \in B_r}\left|\sum_{i=0}^{n_j-1}\psi_\ell(z_i,k_j+i)\right|\\
&\leq N^2\sum_{r=0}^{j-1}n_r \leq \frac{N^2n_j}{2^{j+1}}\end{align*}
and in a similar fashion
\begin{align*}\left|\sum_{i=0}^{n_j-1}\sum_{\ell \in A_j} \psi_\ell(z_i,k_j+i) - \sum_{i=0}^{n_j-1}\sum_{\ell \in C_j} \psi_\ell(z_i,k_j+i)\right| &\leq \sum_{\ell \in C_j \setminus A_j}\left|\sum_{i=0}^{n_j-1}\psi_\ell(z_i,k_j+i)\right|\\
&\leq \sum_{\ell \in \bigcup_{r=0}^{j-1}B_j}\left|\sum_{i=0}^{n_j-1}\psi_\ell(z_i,k_j+i)\right|\\
&\leq N^2\sum_{r=0}^{j-1}n_r \leq \frac{N^2n_j}{2^{j+1}}.\end{align*}
Combining the three estimates above we obtain the inequality
\begin{equation}\label{bombaybynight}\sum_{i=0}^{n_j-1}\psi_A(z_i,k_j+i) \geq\left( \sum_{i=0}^{n_j-1}\sum_{\ell \in C_j}\psi_\ell(z_i,k_j+i) \right)- 2N^2(N-1) - \frac{N^2n_j}{2^j}\end{equation}
for every $j \geq 1$.

It remains to give a lower bound for the sum on the right-hand side of \eqref{bombaybynight}. Let $\ell \in C_j$. Since $\psi(a,k)=0$ when $k< \ell$ and when $k\geq\ell+N$ we have
\[\sum_{i=0}^{n_j-1}\psi_\ell(z_i,k_j+i)=\sum_{i=\ell-k_j}^{\ell-k_j+N-1}\psi_\ell(z_i,k_j+i)=\sum_{i=0}^{N-1} \psi_\ell(z_{\ell-k_j+i},\ell+i),\]
where we have used the fact that $0 \leq \ell-k_j \leq \ell-k_j+N-1 \leq n_j-1$ by the definition of $C_j$. 
By the definition of $C_j$ we also know that $\sigma^{\ell-k_j}z \in [\omega]$, which is to say that $z_{\ell-k_j+i}=\omega_i$ for $i=0,\ldots,N-1$, so by the definition of $\psi_\ell$ the sum displayed above is equal to precisely $N$. It follows that
\[\sum_{i=0}^{n_j-1}\sum_{\ell \in C_j } \psi_\ell(z_i,k_j+i) =\sum_{\ell \in C_j }N = N\sum_{i=0}^{n_j-N} \chi_{[\omega]}(\sigma^iz)\]
and so by \eqref{bombaybynight}
\[\sum_{i=0}^{n_j-1}\psi_A(z_i,k_j+i)  \geq  N\sum_{i=0}^{n_j-N} \chi_{[\omega]}(\sigma^iz) - 2N^2(N-1) -\frac{N^2n_j}{2^j}\]
for every $j \geq 1$. Combining this result with \eqref{eq:z-good-for-chi} yields
\[\liminf_{j \to \infty} \frac{1}{n_j}\sum_{i=0}^{n_j-1}\psi(z_i,k_j+i) \geq \lim_{j \to \infty} \frac{N}{n_j}\sum_{i=0}^{n_j-1}\chi_{[\omega]}(\sigma^iz) \geq N\mu([\omega]),\]
and in combination with \eqref{eq:k-good-for-phi} we deduce that for every real number $\lambda>0$
\begin{align*}\log\varrho\left(L_0^{\phi+\lambda\psi},L_1^{\phi+\lambda \psi}\right) &=\lim_{n \to \infty}\sup_{x \in \Sigma_2}\sup_{k \in \mathbb{Z}} \frac{1}{n}\sum_{i=0}^{n-1}(\phi+\lambda\psi)(x_i,k+i)\\
&\geq \liminf_{j \to \infty} \frac{1}{n_j}\sum_{i=0}^{n_j-1} (\phi+\lambda\psi)(z_i,k_j+i)\\
&\geq \log\varrho\left(L_0^\phi,L_1^\phi\right)+\lambda N\mu([\omega]).\end{align*}
This completes the proof that $\psi:=\psi_A$ satisfies (ii) and hence completes the proof of the lemma.

\section{Acknowledgments}

The author was supported by EPSRC grant EP/L026953/1.

\bibliographystyle{siam}
\bibliography{Gurvits-bib}
\end{document}